\documentclass{article}
\usepackage[utf8]{inputenc}
\usepackage{slashed}
\usepackage{amsmath}
\usepackage{amsfonts}
\usepackage{mathtools}
\usepackage{amssymb}
\usepackage{authblk}
\usepackage{amsthm}
\newtheorem{theorem}{Theorem}[section]
\newtheorem{lemma}{Lemma}[section]
\newtheorem{proposition}{Proposition}[section]
\newtheorem{corollary}[proposition]{Corollary}
\newtheorem{remark}{Remark}[section]
\newtheorem{conjecture}{Conjecture}[section]

\title{Energy distribution of solutions to defocusing semi-linear wave equation in higher dimensional space $\footnote{MSC classes: 35L05, 35L71.}$}
\author{Liang Li, Ruipeng Shen}
\affil{Centre for Applied Mathematics\\ Tianjin University\\ Tianjin, China}
\date{May 9, 2021}
\usepackage{geometry}
\begin{document}

\maketitle
\begin{abstract}
The topic of this paper is a semi-linear, defocusing wave equation $u_{t t}-\Delta u=-|u|^{p-1} u$ in sub-conformal case in the higher dimensional space whose initial data are radical and come with a finite energy. We prove some decay estimates of the the solutions if initial data decay at a certain rate as the spatial variable tends to infinity. A combination of this property with a method of characteristic lines give a scattering result if the initial data satisfy
$$E_{\kappa}\left(u_{0}, u_{1}\right)=\int_{\mathbb{R}^{d}}\left(|x|^{\kappa}+1\right)\left(\frac{1}{2}\left|\nabla u_{0}(x)\right|^{2}+\frac{1}{2}\left|u_{1}(x)\right|^{2}+\frac{1}{p+1}\left|u_{0}(x)\right|^{p+1}\right) d x<+\infty.$$
Here $\kappa=\frac{(2-d)p+(d+2)}{p+1}$.
\end{abstract}
\section{Introduction}
\subsection{Background}
In this work we consider the defocusing nonlinear wave equation in dimensions $d\geq 3$.
$$
\left\{\begin{array}{ll}
\partial_{t}^{2} u-\Delta u=-|u|^{p-1} u, \quad(x, t) \in \mathbb{R}^{d} \times \mathbb{R} ; \\
u(\cdot, 0)=u_{0} ; & (CP1)\\
u_{t}(\cdot, 0)=u_{1} . & 
\end{array}\right.
$$
The conserved energy is defined by
$$
E\left(u, u_{t}\right)=\int_{\mathbb{R}^{d}}\left(\frac{1}{2}|\nabla u(x, t)|^{2}+\frac{1}{2}\left|u_{t}(x, t)\right|^{2}+\frac{1}{p+1}|u(x, t)|^{p+1}\right) d x
.$$
\textbf{Local theory} \quad Defocusing nonlinear wave equations$$
\partial_{t}^{2} u-\Delta u=-|u|^{p-1} u, \quad(x, t) \in \mathbb{R}^{d} \times \mathbb{R}
$$
have been extensively studied, especially in the 3 or higher dimensional space. The existence and uniqueness of solutions to semi-linear wave equation like (CP1) follows a combination of suitable Strichartz estimates and a fixed-point argument, Kapitanski \cite{ref1} and Lindblad-Sogge \cite{ref2} give more details.\\\\
\begin{conjecture}
Any solution to (CP1) with initial data $(u_{0},u_{1})\in \dot{H}^{s_{p}} \times \dot{H}^{s_{p}-1}$ must exist for all time $t\in \mathbb{R}$ and scatter in both two time directions.
\end{conjecture}
This is still an open problem. Although there are many related results by different methods.\\
\textbf{Scattering results with a priori estimates}\quad
It has been proved that if a solution $u$ with a maximal lifespan I satisfies an a priori estimate $$
\sup _{t \in I}\left\|\left(u(\cdot, t), u_{t}(\cdot, t)\right)\right\|_{\dot{H}^{s_{p}} \times \dot{H}^{s_{p}-1}\left(\mathbb{R}^{d}\right)}<+\infty
,
$$
then $u$ is defined for all time $t$ and scatters. In fact, there are many works for different ranges of $d$ and $p$, sometimes with a radial assumption. Different methods were used for different range of $d$ and $p$. The details can be found in Kenig-Merle \cite{ref4}, Killip-Visan \cite{ref14} (3 dimension), Killip-Visan \cite{ref15} (all dimensions) for energy supercritical case and  R.Shen \cite{ref5} , Dodson-Lawrie \cite{ref6} (3 dimension), Rodriguez \cite{ref16} (dimension 4 and 5) for energy subcritical case.\\\\
\textbf{Strong Assumption on initial data}\quad
There are multiple scattering results if we assume that the initial data satisfy stronger regularity and/or decay conditions. These results are usually proved via a suitable global space-time integral estimate. In $d\geq 3$ case, if the initial data $(u_0,u_1)$ satisfy 
$$ \int_{\mathbb{R}^{d}}(1+|x|)^{2}\left(\frac{1}{2}\left|\nabla u_{0}(x)\right|^{2}+\frac{1}{2}\left|u_{1}(x)\right|^{2}+\frac{1}{p+1}\left|u_{0}(x)\right|^{p+1}\right)<+\infty
$$the conformal conservation law (see Ginibre-Velo \cite{ref18} and Hidano \cite{ref19}) leads to the scattering of solutions for $1+4/(d-1)\leq p< 1+4/(d-2).$ In 3-dimension case, R.Shen \cite{ref17} proved the scattering result for $3\leq p\leq 5$ if initial data $\left(u_{0}, u_{1}\right)$ are radial and satisfy $$
\int_{\mathbb{R}^{3}}(|x|^{\kappa}+1)\left(\frac{1}{2}\left|\nabla u_{0}\right|^{2}+\frac{1}{2}\left|u_{1}\right|^{2}+\frac{1}{p+1}|u_{0}|^{p+1}\right) d x<+\infty,
$$ here $\kappa>\kappa_{0}(p)=\frac{5-p}{p+1}$. In $d=3,p=3$ case, Dodson \cite{ref10} gives a proof of the conjecture above for (CP1) with radial data.\\
But most of these results above are in the conformal case or super-conformal case.
\subsection{Main tools}
Next we introduce main tools of this paper. The first tool is still a Morawetz-type estimate, and the second tool is the method of characteristic lines.\\
\textbf{Morawetz estimates}\quad
This kind of estimates were first found by Morawetz \cite{ref3} for wave/Klein-Gordon equations. Lin-Strauss \cite{ref7} then generalized Morawetz estimates to Schr{\"o}ndinger equations. Coliiander-Keel-Staffilani-Takaoka-Tao \cite{ref8} introduced interaction Morawetz estimates for Schr{\"o}ndinger equations. Nowadays the Morawetz estimate has been one of the most important tools in the study of dispersive equations.\\\\
\textbf{Method of characteristic lines} \quad R.Shen \cite{ref9} generalizes the 3D method to higher dimensions. Let $u$ be a radial solution to (CP1) with a finite energy, reduce the equation to a one-dimensional one by defining $w(r, t)=r^{\frac{d-1}{2}} u(r, t)$, and considering the equation that $w$ satisfies $$\left(\partial_{t}+\partial_{r}\right)\left(w_{t}-w_{r}\right)=\partial_{t}^{2} w-\partial_{r}^{2} w=-\frac{(d-1)(d-3)}{4} d r^{\frac{d-5}{2}} u-r^{\frac{d-1}{2}}|u|^{p-1} u.$$
This enable us to evaluate the variation of $w_{t} \pm w_{r}$  along characteristic lines $t \mp r=Const$ and obtain plentiful information about the asymptotic behaviour of solutions.

\subsection{Main Results}
Now we give the main results of this work. Throughout this paper we always assume $d\geq 3$ and $p<1+\frac{4}{d-1}$. In this case it is well known that if the initial data come with a finite energy, then the solution exist for all time. Please refer to Ginibre-Velo \cite{ref21}. Our first result is about the energy distribution of solutions to (CP1).
\begin{theorem}
Assume $d\geq3,\,1+\frac{4}{d-1}>p>1+\frac{2}{d-1}$. Let $u$ be a solution to (CP1) with a finite energy. Then\\
(a) The following limits hold as time tends to infinity $$
\lim _{t \rightarrow \pm \infty} \int_{|x|<|t|} \frac{|t|-|x|}{|t|} e(x, t) d x=0.
$$
(b) The inward/outward part of energy vanishes as time tends positive/negative infinity.
$$
\lim _{t \rightarrow \pm \infty} \int_{\mathbb{R}^{d}}\left(\left|u_{r} \pm u_{t}\right|^{2}+|\nabla\mkern-13mu/ u|^{2}+|u|^{p+1}\right) d x=0.
$$
(c) Furthermore, if the initial data satisfy $E_{\kappa}\left(u_{0}, u_{1}\right)<+\infty$ for a constant $$\kappa \in \left(0,\frac{(d-1)(p-1)-2}{2}\right),\qquad if \,\, 1+\frac{4}{d-1}>p>1+\frac{2}{d-1};$$
\qquad then we have the following decay estimates$$
\lim _{t \rightarrow \pm \infty} \int_{|x|<|t|} \frac{|t|-|x|}{|t|^{1-\kappa}} e(x, t) d x=0;
$$$$
\lim _{t \rightarrow \pm \infty}|t|^{\kappa} \int_{\mathbb{R}^{d}}\left(\left|u_{r} \pm u_{t}\right|^{2}+\left.|\nabla\mkern-13mu/ u\right|^{2}+|u|^{p+1}\right) d x=0.
$$\end{theorem}
As an application of the theory on energy distribution, we also prove the following scattering result.
\begin{theorem}
Assume $d\geq3,\, 1+\frac{4}{d-1}>p>\frac{3-d+2\sqrt{d^2-d+1}}{d-1}$. Let $u$ be a radial solution to (CP1) with a finite energy and the initial data $(u_{0}, u_{1})$ satisfy $E_{\frac{(2-d)p+(d+2)}{p+1}}\left(u_{0}, u_{1}\right)<+\infty$, then the solution $u$ scatters in both two time directions. More precisely, there exist two radial finite-energy free waves $\tilde{u}^{+}, \tilde{u}^{-}$, so that 
$$
\lim _{t \rightarrow \pm \infty}\left\|\left(\tilde{u}^{\pm}(\cdot, t)-u(\cdot, t), \tilde{u}_{t}^{\pm}(\cdot, t)-u_{t}(\cdot, t)\right)\right\|_{\dot{H}^{1} \times L^{2}\left(\mathbb{R}^{d}\right)}=0.
$$
Here $E_{\frac{(2-d)p+(d+2)}{p+1}}\left(u_{0}, u_{1}\right)=\int_{\mathbb{R}^{d}}\left(|x|^{\frac{(2-d)p+(d+2)}{p+1}}+1\right)\left(\frac{1}{2}\left|\nabla u_{0}(x)\right|^{2}+\frac{1}{2}\left|u_{1}(x)\right|^{2}+\frac{1}{p+1}\left|u_{0}(x)\right|^{p+1}\right) d x.$
\end{theorem}
\subsection{The Structure of This Paper}
This paper is organized as follows. We first give a few preliminary results in section 2, In section 3 we give a Morewetz identity and a Morawetz inequality, which is the main tool of this paper. Next in section 4 we prove the energy distribution properties of the solutions. Finally we prove the scattering theory of the solution $u$ under an additional decay assumption in the last section.
\section{Preliminary Results}
\textbf{Notations}\quad In this work we will use the notation $e(x,t)$ for the energy density
$$
e(x, t)=\frac{1}{2}|\nabla u(x, t)|^{2}+\frac{1}{2}\left|u_{t}(x, t)\right|^{2}+\frac{1}{p+1}|u(x, t)|^{p+1}.
$$
We use $u_{r}, \nabla\mkern-13mu/ u$ for the derivative in the radial direction and the covariant derivative on the sphere centred at the origin, respectively;
$$
u_{r}(x, t)=\frac{x}{|x|} \cdot \nabla u(x, t) ; \quad {\nabla\mkern-13mu/ u}=\nabla u-u_{r} \frac{x}{|x|} ; \quad|\nabla u|^{2}=\left|u_{r}\right|^{2}+|\nabla\mkern-13mu/ u|^{2}.
$$
We also define the weighted energy $$
E_{\kappa}\left(u_{0}, u_{1}\right)=\int_{\mathbb{R}^{d}}\left(|x|^{\kappa}+1\right)\left(\frac{1}{2}\left|\nabla u_{0}(x)\right|^{2}+\frac{1}{2}\left|u_{1}(x)\right|^{2}+\frac{1}{p+1}\left|u_{0}(x)\right|^{p+1}\right) d x .
$$
In this work $\sigma_{R}$ represents the regular measure of the sphere $\left\{x \in \mathbb{R}^{d}:|x|=R\right\}$. We also define $c_{d}$ to be the area of the united sphere $\mathbb{S}^{d-1}$. Thus we have the following identities for any radial function $f(x)$
$$
\int_{|x|=r} f(x) d \sigma_{r}(x)=c_{d} r^{d-1} f(r) ; \quad \int_{\mathbb{R}^{d}} f(x) d x=c_{d} \int_{0}^{\infty} f(r) r^{d-1} d r .
$$
The notation $A\lesssim B$ means that there exists a constant $c$, so that the inequality $A\leq c B$ holds. We may also put subscript($s$) to indicate that the constant $c$ depends on the given subscript($s$) but nothing else. In particular, the symbol $\lesssim_{1}$ is used if $c$ is an absolute constant.
\begin{lemma}
(Pointwise Estimate) Assume $d\geq 3$. all radial $\dot{H}^{1}\left(\mathbb{R}^{d}\right)$ functions $u$ satisfy$$
|u(r)| \lesssim_{d}   r^{-\frac{d-2}{2}}\|u\|_{\dot{H}^{1}}, \quad r>0.
$$If $u$ also satisfy $u \in L^{p+1}\left(\mathbb{R}^{d}\right)$, then its decay is stronger as $r \rightarrow+\infty$.$$
|u(r)|\lesssim_{d}     r^{-\frac{2(d-1)}{p+3}}\|u\|_{\dot{H}^{1}}^{\frac{2}{p+3}}\|u\|_{L^{p+1}}^{\frac{p+1}{p+3}}, \quad r>0.
$$\end{lemma}
This lemma has been known for many years, more details can be found, for example, in C.E.Kenig and F.Merle \cite{ref4} for $d=3$ and R.Shen \cite{ref9} for $d\geq3$.
\begin{lemma}
(Radiation Field) Assume that $d\geq 3$ and let $u$ be a solution to the free wave equation $\partial_{t}^{2} u-\Delta u=0$ with initial data $\left(u_{0}, u_{1}\right) \in H^{1} \times L^{2}\left(\mathbb{R}^{d}\right)$. Then$$
\lim _{t \rightarrow+\infty} \int_{\mathbb{R}^{d}}\left(|\nabla u(x, t)|^{2}-\left|u_{r}(x, t)\right|^{2}+\frac{|u(x, t)|^{2}}{|x|^{2}}\right) d x=0,
$$and there exists a unique function $G_{+} \in L^{2}\left(\mathbb{R} \times \mathbb{S}^{d-1}\right)$ such that$$
\lim _{t \rightarrow+\infty} \int_{0}^{\infty} \int_{\mathbb{S}^{d-1}}\left|r^{\frac{d-1}{2}} \partial_{t} u(r \theta, t)-G_{+}(r-t, \theta)\right|^{2} d \theta d r=0;
$$$$
\lim _{\iota \rightarrow+\infty} \int_{0}^{\infty} \int_{\mathbb{S}^{d-1}}\left|r^{\frac{d-1}{2}} \partial_{r} u(r \theta, t)+G_{+}(r-t, \theta)\right|^{2} d \theta d r=0.
$$
In addition, the map$$
\left(u_{0}, u_{1}\right) \rightarrow \sqrt{2} G_{+}
$$$$\quad
\dot{H}^{1} \times L^{2}\left(\mathbb{R}^{d}\right) \rightarrow L^{2}\left(\mathbb{R} \times \mathbb{S}^{d-1}\right)
$$ 
is a bijective isometry.
\end{lemma}
This result was known many years ago, please see Friedlander \cite{ref11,ref12}. Duyckaerts-Kenig-Merle \cite{ref13} gives a proof for all dimensions $d\geq3$.
\section{Morawetz identity and Morewetz inequality}
\subsection{Morawetz identity}
\begin{proposition}
(Morawetz identity) Let $u$ be a solution to (CP1) with a finite energy E. Then the following identity holds for any $R>0$ and time $t_{1}<t_{2}$.
$$
\begin{aligned}
& \frac{1}{2 R} \int_{t_{1}}^{t_{2}} \int_{|x|<R}\left(|\nabla u|^{2}+\left|u_{t}\right|^{2}+\frac{(d-1)(p-1)-2}{p+1}|u|^{p+1}\right) d x d t+\frac{d-1}{4 R^{2}} \int_{t_{1}}^{t_{2}} \int_{|x|=R}|u|^{2} d \sigma_{R}(x) d t \\
+& \int_{t_{1}}^{t_{2}} \int_{|x|>R}\left(\frac{|\nabla\mkern-13mu/ u|^{2}}{|x|}+\frac{(d-1)(p-1)}{2(p+1)} \cdot \frac{|u|^{p+1}}{|x|}+\frac{(d-3)(d-1)}{4} \cdot \frac{|u|^{2}}{|x|^{3}}\right) d x d t \\
+&\left.\sum_{i=1,2} \int_{|x|<R}\left(\frac{R^{2}-|x|^{2}}{2 R^{2}}\left|u_{r}\right|^{2}+\frac{1}{2}\left|\frac{|x|}{R} u_{r}+\frac{(d-1) u}{2 R}+(-1)^{i} u_{t}\right|^{2}+\frac{\left(d^{2}-1\right)|u|^{2}}{8 R^{2}}+\frac{|\nabla\mkern-13mu/ u |^{2}}{2}+\frac{|u|^{p+1}}{p+1}\right)\right|_{t=t_{i}} d x \\
+&\left.\sum_{i=1,2} \int_{|x|>R}\left(\frac{1}{2}\left|u_{r}+\frac{d-1}{2} \frac{u}{|x|}+(-1)^{i} u_{t}\right|^{2}+\frac{|\nabla\mkern-13mu/ u|^{2}}{2}+\frac{|u|^{p+1}}{p+1}+\frac{(d-1)(d-3)\left|u\left(x, t_{i}\right)\right|^{2}}{8|x|^{2}}\right)\right|_{t=t_{i}} d x=2 E .
\end{aligned}
$$
\end{proposition}
\begin{proof} we follow a similar argument to the given by Perthame and Vega in the final section of their work \cite{ref20}. Let us first consider solutions with compact support. Given a positive constant $R$, we define two radical function $\Psi$ and $\varphi$ by
$$
\nabla \Psi=\left\{\begin{array}{ll}
x, & \text { if }|x| \leq R ; \\
R x /|x|, & \text { if }|x| \geq R ;
\end{array} \quad \varphi=\left\{\begin{array}{ll}
1 / 2, & \text { if }|x| \leq R ;\\
0, & \text { if }|x|>R.
\end{array}\right.\right. 
$$
Since $u$ is defined for all time $t$, we may also define a function on $R$
$$
\mathcal{E}(t)=\int_{\mathbb{R}^{d}} u_{t}(x, t)\left(\nabla u(x, t) \cdot \nabla \Psi+u(x, t)\left(\frac{\Delta \Psi}{2}-\varphi\right)\right) d x. \eqno{(1)}
$$
We may differentiate $\mathcal{E}$, utilize the equation $$
u_{t t}-\Delta u=-|u|^{p-1} u ,
$$apply integration by parts and obtain
$$
\begin{aligned}
-\mathcal{E}^{\prime}(t)=& \int_{\mathbb{R}^{d}}\left(\sum_{i, j=1}^{d} u_{i} \Psi_{i j} u_{j}-\varphi|\nabla u|^{2}+\varphi\left|u_{t}\right|^{2}\right) d x+\frac{1}{4} \int_{\mathbb{R}^{d}} \nabla\left(|u|^{2}\right) \cdot \nabla(\Delta \Psi-2 \varphi) d x \\
&+\int_{\mathbb{R}^{d}}|u|^{p+1}\left(\frac{p-1}{2(p+1)} \Delta \Psi-\varphi\right) d x \\
=& I_{1}+I_{2}+I_{3}.
\end{aligned}\eqno{(2)}
$$
Here we have
$$
\begin{array}{l}
\Psi_{i j}=\left\{\begin{array}{ll}
\delta_{i j}, & \text { if }|x|<R ; \\
\frac{R \delta_{i j}}{|x|}-\frac{R x_{i} x_{j}}{|x|^{3}}, & \text { if }|x|>R ;
\end{array} \quad \Delta \Psi=\left\{\begin{array}{ll}
d, & \text { if }|x|<R ; \\
R (d-1)/|x|, & \text { if }|x|>R ;
\end{array}\right.\right. \\\\
\Delta \Psi-2 \varphi=\left\{\begin{array}{ll}
d-1, & \text { if }|x| \leq R ; \\
R (d-1)/|x|, & \text { if }|x| \geq R ;
\end{array} \in C\left(\mathbb{R}^{d}\right)\right..
\end{array}
$$
When $|x|>R$, we may calculate$$
\sum_{i, j=1}^{d} u_{i} \Psi_{i j} u_{j}=\sum_{i, j=1}^{d} u_{i}\left(\frac{R \delta_{i j}}{|x|}-\frac{R x_{i} x_{j}}{|x|^{3}}\right) u_{j}=\frac{R}{|x|}|\nabla u|^{2}-\frac{R|\nabla u \cdot x|^{2}}{|x|^{3}}=\frac{R}{|x|}|\nabla\mkern-13mu/ u|^{2}.\eqno{(3)}
$$
Thus we have$$
I_{1}=\frac{1}{2} \int_{|x|<R}\left(|\nabla u|^{2}+\left|u_{t}\right|^{2}\right) d x+R \int_{|x|>R} \frac{\left.|\nabla\mkern-13mu/ u\right|^{2}}{|x|} d x.
\eqno{(4)}$$
The last term in the equality above, a basic computation shows$$
I_{3}=\frac{(d-1)(p-1)-2}{2(p+1)} \int_{|x|<R}|u|^{p+1} d x+\frac{(p-1)(d-1) R}{2(p+1)} \int_{|x|>R} \frac{|u|^{p+1}}{|x|} d x.
\eqno{(5)}$$
Let us calculate the left hand carefully
$$
\begin{aligned}
I_{2} &=\frac{1}{4} \int_{\mathbb{R}^{d}} \nabla\left(|u|^{2}\right) \cdot \nabla(\Delta \Psi-2 \varphi) d x \\
&=\frac{1}{4} \int_{|x|>R} \nabla\left(|u|^{2}\right) \cdot \frac{-R(d-1) x}{|x|^{3}} d x \\
&=\frac{1}{4} \int_{|x|>R}\left[\operatorname{div}\left(|u|^{2} \cdot \frac{-R(d-1) x}{|x|^{3}}\right)+(d-3)(d-1) \frac{R}{|x|^{3}}|u|^{2}\right] d x \\
&=\frac{d-1}{4 R} \int_{|x|=R}|u|^{2} d \sigma_{R}(x)+\frac{(d-3)(d-1)}{4} \int_{|x|>R} \frac{|u|^{2}}{|x|^{3}} d x.
\end{aligned}
\eqno{(6)}
$$
Since$-\mathcal{E}^{\prime}(t)=I_{1}+I_{2}+I_{3}$, we have\\
$$\int_{t_{1}}^{t_{2}}\left(I_{1}+I_{2}+I_{3}\right) d t=\mathcal{E}\left(t_{1}\right)-\mathcal{E}\left(t_{2}\right).\eqno{(7)}
$$\\
We rewrite in the form of
$$
\begin{aligned}
R \mathcal{E}\left(t_{1}\right)=& \int_{\mathbb{R}^{d}} R u_{t}(x, t)\left(\nabla u(x, t) \cdot \nabla \Psi+u(x, t)\left(\frac{\Delta \Psi}{2}-\varphi\right)\right) d x \\
=& \frac{1}{2} \int_{\mathbb{R}^{d}}\left(R^{2}\left|u_{t}\left(x, t_{1}\right)\right|^{2}+\left|\nabla u\left(x, t_{1}\right) \cdot \nabla \Psi+u\left(x, t_{1}\right)\left(\frac{\Delta \Psi}{2}-\varphi\right)\right|^{2}\right) d x \\
& \quad-\frac{1}{2} \int_{\mathbb{R}^{d}}\left|\nabla u\left(x, t_{1}\right) \cdot \nabla \Psi+u\left(x, t_{1}\right)\left(\frac{\Delta \Psi}{2}-\varphi\right)-R u_{t}\left(x, t_{1}\right)\right|^{2} d x \\
=& J_{1}-J_{2}.
\end{aligned}\eqno{(8)}
$$
Then we calculate $J_{1},J_{2}$ 
$$
\begin{aligned}
J_{1} &=\frac{1}{2} \int_{\mathbb{R}^{d}}\left(R^{2}\left|u_{t}\right|^{2}+|\nabla u \cdot \nabla \Psi|^{2}+\left(\frac{\Delta \Psi}{2}-\varphi\right) \nabla\left(|u|^{2}\right) \cdot \nabla \Psi+\left(\frac{\Delta \Psi}{2}-\varphi\right)^{2}|u|^{2}\right) d x \\
&=\frac{1}{2} \int_{\mathbb{R}^{d}}\left[R^{2}\left|u_{t}\right|^{2}+|\nabla u \cdot \nabla \Psi|^{2}-\operatorname{div}\left(\left(\frac{\Delta \Psi}{2}-\varphi\right) \nabla \Psi\right)|u|^{2}+\left(\frac{\Delta \Psi}{2}-\varphi\right)^{2}|u|^{2}\right] d x.
\end{aligned}\eqno{(9)}
$$
A basic calculation  shows$$
\operatorname{div}\left(\left(\frac{\Delta \Psi}{2}-\varphi\right) \nabla \Psi\right)=\left\{\begin{array}{ll}
d(d-1)/2, & \text { if } |x|<R ;\\
\frac{(d-1)(d-2)R^2}{2|x|^2}, & \text { if } |x|>R.
\end{array}\right.\eqno{(10)}
$$
Thus we have
$$
\begin{aligned}
J_{1}=& \frac{1}{2} \int_{|x|<R}\left[R^{2}\left|u_{t}\right|^{2}+|x \cdot \nabla u|^{2}+\frac{1-d^{2}}{4}|u|^{2}\right] d x \\
&+\frac{1}{2} \int_{|x|>R}\left[R^{2}\left|u_{t}\right|^{2}+R^{2}\left|u_{r}\right|^{2}+\frac{R^{2}(d-1)(3-d)|u|^{2}}{4|x|^{2}}\right] d x \\
=& R^{2} E-R^{2} \int_{|x|<R}\left[\frac{R^{2}-|x|^{2}}{2 R^{2}}\left|u_{r}\right|^{2}+\frac{d^{2}-1}{8 R^{2}}|u|^{2}+\frac{1}{2}|\nabla\mkern-13mu/ u|^{2}+\frac{1}{p+1}|u|^{p+1}\right] d x \\
&+\frac{(d-1)(3-d) R^{2}}{8} \int_{|x|>R} \frac{|u|^{2}}{|x|^{2}} d x-R^{2} \int_{|x|>R}\left(\frac{1}{2}|\nabla\mkern-13mu/ u|^{2}+\frac{1}{p+1}|u|^{p+1}\right) d x.
\end{aligned}\eqno{(11)}
$$
In addition we have
$$
J_{2}=\frac{1}{2} \int_{|x|<R}\left|x \cdot \nabla u+\frac{d-1}{2} u-R u_{t}\right|^{2} d x+\frac{R^{2}}{2} \int_{|x|>R}\left|\frac{x}{|x|} \cdot \nabla u+\frac{d-1}{2}\frac{u}{|x|}-u_{t}\right|^{2} d x. \eqno{(12)}
$$
Combining $J_{1},J_{2}$, we obtain
$$\begin{aligned}
R \mathcal{E}\left(t_{1}\right)=& R^{2} E-R^{2} \int_{|x|>R}\left(\frac{1}{2}\left|u_{r}+\frac{d-1}{2}\frac{u}{|x|}-u_{t}\right|^{2}+\frac{\left.|\nabla\mkern-13mu/ u\right|^{2}}{2}+\frac{|u|^{p+1}}{p+1}+\frac{(d-1)(d-3)}{8}\frac{|u|^{2}}{|x|^{2}}\right) d x\\
& -R^{2} \int_{|x|<R}\left[\frac{R^{2}-|x|^{2}}{2 R^{2}}\left|u_{r}\right|^{2}+\frac{1}{2}\left|\frac{|x|}{R} u_{r}+\frac{d-1}{2 R}u-u_{t}\right|^{2}+\frac{\left.|\nabla\mkern-13mu/ u\right|^{2}}{2}+\frac{|u|^{p+1}}{p+1}+\frac{(d^2-1)|u|^{2}}{8 R^{2}}\right] d x.
\end{aligned}
$$
Finally we find a similar expression of $-R \mathcal{E}\left(t_{2}\right)$
$$
\begin{aligned}
-R \mathcal{E}\left(t_{2}\right)=& R^{2} E-R^{2} \int_{|x|>R}\left(\frac{1}{2}\left|u_{r}+\frac{(d-1)u}{2|x|}+u_{t}\right|^{2}+\frac{\left.|\nabla\mkern-13mu/ u\right|^{2}}{2}+\frac{|u|^{p+1}}{p+1}+\frac{(d-1)(d-3)|u|^{2}}{8|x|^{2}}\right) d x\\
& -R^{2} \int_{|x|<R}\left[\frac{R^{2}-|x|^{2}}{2 R^{2}}\left|u_{r}\right|^{2}+\frac{1}{2}\left|\frac{|x|}{R} u_{r}+\frac{(d-1)u}{2 R}+u_{t}\right|^{2}+\frac{\left.|\nabla\mkern-13mu/ u\right|^{2}}{2}+\frac{|u|^{p+1}}{p+1}+\frac{(d^2-1)|u|^{2}}{8 R^{2}}\right] d x.
\end{aligned}
$$
Then plug all the expressions of $I_{1},I_{2},I_{3}$ and $-R \mathcal{E}\left(t_{2}\right),-R \mathcal{E}\left(t_{2}\right)$ in to the integral identity to finish the proof if $u$ is compactly supported. In order to deal with the general solution $u$, we fix a smooth radial cut-off function$\phi: \mathbb{R}^{2} \rightarrow[0,1]$ so that$$
\phi(x)=\left\{\begin{array}{ll}
1, & \text { if }|x| \leq 1 ;\\
0, & \text { if }|x|>2;
\end{array}\right.
$$
define initial data $\left(u_{0, R^{\prime}}(x), u_{1, R^{\prime}}(x)\right)=\phi\left(x / R^{\prime}\right)\left(u\left(x, t_{1}\right), u_{t}\left(x, t_{1}\right)\right)$ and consider the corresponding solution $u_{R^{\prime}}$ to (CP1). The argument above shows that $u_{R^{\prime}}$ satisfies Morawetz identity. We observe\\
1) The identity $u_{R{\prime}}(x,t)=u(x,t)$ holds if $|x|<R^{\prime}+t_1-t$ by finite speed of propagation;\\
2) $E(u_{0,R^{\prime}},u_{1,R^{\prime}})\to E$ as $R^{\prime}\to\infty$.\\
3) The energies of $u_{R{\prime}}$ and $u$ in the region where $u_{R^{\prime}} \neq u$ both converge to zero as $R^{\prime}\to+\infty$ by finite speed of propagation and energy conservation law.\\
These facts enable us to take the limit $R^{\prime}\to+\infty$ and prove Morawetz identity for general solutions $u$.
\end{proof}
\subsection{Morawetz inequalities}
A combination of Morawetz identity and finite speed of propagation gives a few useful inequalities, which is the main tool of this paper. The key observation here is that if $R$ is large, the first term in the Morawetz identity is almost $2E$ when $t_{1}\le-R$ and $t_{2}\le R$, thus all other terms must be small.
\begin{corollary}
Let $u$ be a solution to (CP1) with initial data$\left(u_{0}, u_{1}\right) \in\left(\dot{H}^{1}\left(\mathbb{R}^{d}\right) \cap L^{p+1}\left(\mathbb{R}^{d}\right)\right) \times L^{2}\left(\mathbb{R}^{d}\right)$. Given any $R>0,r\ge0$ we have$$
\sum_{j=1}^{6} M_{j} \leq \int_{\mathbb{R}^{d}} \min \{|x| / R, 1\}\left(\left|\nabla u_{0}\right|^{2}+\left|u_{1}\right|^{2}+\frac{2}{p+1}\left|u_{0}\right|^{p+1}\right) d x.
$$
The notations $M_{j}$ are defined by
$$
M_{1}=\frac{1}{2 R} \int_{R<|t|<R+r} \int_{|x|<R}\left(|\nabla u|^{2}+\left|u_{t}\right|^{2}+\frac{(d-1)(p-1)-2}{p+1}|u|^{p+1}\right) d x dt;\qquad\qquad\qquad\qquad\quad\qquad\qquad\qquad\qquad\quad
$$$$
M_{2}=\frac{(d-1)(p-1)-4}{2(p+1) R} \int_{-R}^{R} \int_{|x|<R}|u|^{p+1} d x d t;\qquad\qquad\qquad\qquad\qquad\qquad\qquad\qquad\qquad\qquad\qquad\qquad\qquad\quad\qquad\qquad\qquad\qquad
$$$$
M_{3}=\frac{d-1}{4 R^{2}} \int_{-R-r}^{R+r} \int_{|x|=R}|u|^{2} d \sigma_{R}(x)dt;\qquad\qquad\qquad\qquad\qquad\qquad\qquad\qquad\qquad\qquad\qquad\qquad\qquad\qquad\quad\qquad\qquad\qquad\qquad
$$$$
M_{4}=\int_{-R-r}^{R+r} \int_{|x|>R}\left(\frac{\left.|\nabla\mkern-13mu/ u\right|^{2}}{|x|}+\frac{(d-1)(p-1)}{2(p+1)}\frac{|u|^{p+1}}{|x|}+\frac{(d-3)(d-1)}{4}\frac{|u|^{2}}{|x|^{3}}\right) d xdt;\qquad\qquad\qquad\qquad\qquad\qquad\qquad\qquad
$$$$
M_{5}=\left.\sum_{\pm} \int_{|x|<R}\left(\frac{R^{2}-|x|^{2}}{2 R^{2}}\left|u_{r}\right|^{2}+\frac{1}{2}\left|\frac{|x|}{R} u_{r}+\frac{(d-1)u}{2 R} \pm u_{t}\right|^{2}+\frac{(d^2-1)|u|^{2}}{8 R^{2}}+\frac{\left.|\nabla\mkern-13mu/ u\right|^{2}}{2}+\frac{|u|^{p+1}}{p+1}\right)\right|_{t=\pm(R+r)} d x;\qquad\qquad\qquad\qquad
$$$$
M_{6}=\left.\sum_{\pm} \int_{|x|>R}\left(\frac{1}{2}\left|u_{r}+\frac{(d-1)u}{2|x|} \pm u_{t}\right|^{2}+\frac{\left.|\nabla\mkern-13mu/ u\right|^{2}}{2}+\frac{1}{p+1}|u|^{p+1}+\frac{(d-1)(d-3)}{8}\frac{|u|^{2}}{|x|^{2}}\right)\right|_{t=\pm(R+r)} d x.\qquad\qquad\qquad\qquad
$$\end{corollary}
\begin{proof} We first choose $t_{1}=-R-r,t_{2}=R+r$ in the Morawetz identity,the first term above can be written as a sum of three terms$$
\frac{1}{2 R} \int_{-R-r}^{R+r} \int_{|x|<R}\left(|\nabla u|^{2}+\left|u_{t}\right|^{2}+\frac{(d-1)(p-1)-2}{p+1}|u|^{p+1}\right) d x\qquad\qquad
$$$$
=M_{1}+\frac{1}{2 R} \int_{-R}^{R} \int_{|x|<R}\left(|\nabla u|^{2}+\left|u_{t}\right|^{2}+\frac{(d-1)(p-1)-2}{p+1}|u|^{p+1}\right) d x\qquad
$$$$
=M_{1}+M_{2}+\frac{1}{2 R} \int_{-R}^{R} \int_{|x|<R}\left(|\nabla u|^{2}+\left|u_{t}\right|^{2}+\frac{2}{p+1}|u|^{p+1}\right) d x.\qquad\qquad\quad
$$
Thus$$
\frac{1}{2 R} \int_{-R}^{R} \int_{|x|<R}\left(|\nabla u|^{2}+\left|u_{t}\right|^{2}+\frac{2}{p+1}|u|^{p+1}\right) d x d t+\sum_{j=1}^{6} M_{j}=2E.
\eqno{(13)}$$
In order to prove the first inequality we only need to show
$$
\begin{aligned}
I &=2 E-\frac{1}{2 R} \int_{-R}^{R} \int_{|x|<R}\left(|\nabla u|^{2}+\left|u_{t}\right|^{2}+\frac{2}{p+1}|u|^{p+1}\right) d x d t \\
& \leq \int_{\mathbb{R}^{d}} \min \{|x| / R, 1\}\left(\left|\nabla u_{0}\right|^{2}+\left|u_{1}\right|^{2}+\frac{2}{p+1}\left|u_{0}\right|^{p+1}\right) d x.
\end{aligned}
$$
This follows energy conservation law and finite speed of propagation of energy$$
\begin{aligned}
I &=\frac{1}{R} \int_{-R}^{R} \int_{|x|>R}\left(\frac{1}{2}|\nabla u|^{2}+\frac{1}{2}\left|u_{t}\right|^{2}+\frac{1}{p+1}|u|^{p+1}\right) d x d t \\
& \leq \frac{1}{R} \int_{-R}^{R} \int_{|x|>R-|t|}\left(\frac{1}{2}\left|\nabla u_{0}\right|^{2}+\frac{1}{2}\left|u_{1}\right|^{2}+\frac{1}{p+1}\left|u_{0}\right|^{p+1}\right) d x d t \\
&=\frac{1}{R} \int_{\mathbb{R}^{d}} \min \{|x|, R\}\left(\left|\nabla u_{0}\right|^{2}+\left|u_{1}\right|^{2}+\frac{2}{p+1}\left|u_{0}\right|^{p+1}\right) d x.
\end{aligned}
$$
\end{proof}
\section{Energy Distribution}
In this section we prove Theorem1.1. It suffices to consider the positive time direction $t>0$, since the wave equation is time-reversible.\\
We choose $R=t,r=0$ in corollary 3.3. The following inequalities hold for large time $t>0$.
$$
\begin{aligned}
& \sum_{\pm} \int_{|x|<t}\left(\frac{t^{2}-|x|^{2}}{2 t^{2}}\left|u_{r}\right|^{2}+\frac{1}{2}\left|\frac{|x|}{t} u_{r}+\frac{(d-1) u}{2 t} \pm u_{t}\right|^{2}+\frac{\left(d^{2}-1\right)|u|^{2}}{8 t^{2}}+\frac{|\nabla\mkern-13mu/ u|^{2}}{2}+\frac{|u|^{p+1}}{p+1}\right) d x \\
&+\sum_{\pm} \int_{|x|>t}\left(\frac{1}{2}\left|u_{r}+\frac{(d-1) u}{2|x|} \pm u_{t}\right|^{2}+\frac{|\nabla\mkern-13mu/u|^{2}}{2}+\frac{1}{p+1}|u|^{p+1}+\frac{(d-1)(d-3)}{8} \frac{|u|^{2}}{|x|^{2}}\right) d x \\
\leq & \frac{4-(d-1)(p-1)}{2(p+1) t} \int_{-t}^{t} \int_{|x|<t}\left|u\left(x, t^{\prime}\right)\right|^{p+1} d x d t^{\prime}+\int_{\mathbb{R}^{d}} \min \{|x| / R, 1\}\left(\left|\nabla u_{0}\right|^{2}+\left|u_{1}\right|^{2}+\frac{2}{p+1}\left|u_{0}\right|^{p+1}\right) d x .
\end{aligned}\eqno{(14)}
$$
We observe that
$$
|\nabla u|^{2}+\left|u_{t}\right|^{2} \lesssim_{1}\left|u_{r}\right|^{2}+\left|\frac{|x|}{t} u_{r}+\frac{(d-1)u}{2 t}+u_{t}\right|^{2}+\frac{d^{2}-1}{8}\frac{|u|^{2}}{t^{2}}+\left.|\nabla\mkern-13mu/ u\right|^{2}
$$
and$$
\left|u_{r}+u_{t}\right|^{2} \lesssim_{1} \frac{(t-|x|)^{2}}{t^{2}}\left|u_{r}\right|^{2}+\left|\frac{|x|}{t} u_{r}+\frac{(d-1)u}{2 t}+u_{t}\right|^{2}+\frac{d^{2}-1}{8}\frac{|u|^{2}}{t^{2}}, \quad \text { if }\,\,|x|<t ;
$$$$
\left|u_{r}+u_{t}\right|^{2} \lesssim_{1}\left|u_{r}+\frac{(d-1)u}{2|x|}+u_{t}\right|^{2}+\frac{(d-1)(d-3)}{8}\frac{|u|^{2}}{|x|^{2}},\qquad\qquad \qquad \text { if }\,\,|x|>t .
$$
Then obtain
$$
\frac{1}{p+1} \int_{\mathbb{R}^{d}}\left(|u(x, t)|^{p+1}+|u(x,-t)|^{p+1}\right) d x+c_{1} \sum_{\pm} \int_{|x|<t} \frac{t-|x|}{t}\left(|\nabla u(x, \pm t)|^{2}+\left|u_{t}(x, \pm t)\right|^{2}\right) d x
$$$$
+c_{2} \sum_{\pm} \int_{\mathbb{R}^{d}}\left(\left|\left(u_{r} \pm u_{t}\right)(x, \pm t)\right|^{2}+|\nabla\mkern-13mu/ u(x, \pm t)|^{2}\right) d x\qquad\qquad\qquad\qquad\qquad\qquad
$$$$
\leq \frac{4-(d-1)(p-1)}{2(p+1) t} \int_{-t}^{t} \int_{|x|<t}\left|u\left(x, t^{\prime}\right)\right|^{p+1} d x d t^{\prime}+2 \int_{\mathbb{R}^{d}} \min \{|x| / t, 1\} e(x, 0) d x.
\qquad\qquad\qquad$$
Here $c_{1},c_{2}>0$ are absolute constants. For convenience we introduce the notation
$$
\begin{aligned}
Q(t)=& \frac{1}{p+1} \int_{\mathbb{R}^{d}}\left(|u(x, t)|^{p+1}+|u(x,-t)|^{p+1}\right) d x+c_{1} \sum_{\pm} \int_{|x|<t} \frac{t-|x|}{t}\left(|\nabla u(x, \pm t)|^{2}+\left|u_{t}(x, \pm t)\right|^{2}\right) d x \\
& +c_{2} \sum_{\pm} \int_{\mathbb{R}^{d}}\left(\left|\left(u_{r} \pm u_{t}\right)(x, \pm t)\right|^{2}+|\nabla\mkern-13mu/ u(x, \pm t)|^{2}\right) d x.
\end{aligned}
$$
Then the inequality above implies that Q(t) satisfies the recurrence formula$$
Q(t) \leq \frac{\lambda}{t} \int_{0}^{t} Q\left(t^{\prime}\right) d t^{\prime}+2 \int_{\mathbb{R}^{d}} \min \{|x| / t, 1\} e(x, 0) d x.
$$
Here$$
0<\lambda=\frac{4-(d-1)(p-1)}{2}<1.
$$\\
\textbf{Proof of part (a)(b)}\quad We may rewrite the recurrence formula as
$$
Q(t) \leq \frac{\lambda}{t} \int_{0}^{t} Q\left(t^{\prime}\right) d t^{\prime}+o(1).
$$
We may take upper limits of both sides and obtain an inequality 
$$
\limsup _{t \rightarrow+\infty} Q(t) \leq \limsup _{t \rightarrow+\infty} \frac{\lambda}{t} \int_{0}^{t} Q\left(t^{\prime}\right) d t^{\prime} \leq \lambda \limsup _{t \rightarrow+\infty} Q(t).
$$
We recall the fact $\lambda\in(0,1)$ and observe that $Q(t)\lesssim E$ is uniformly bounded, therefore we have 
$$
\limsup _{t \rightarrow+\infty} Q(t)=0 .
$$
This verifies (a)(b).\\\\
\textbf{Proof of part (c)}\quad Now we assume that initial data satisfy additional decay assumption. We start by multiplying both sides by $t^{k-1}$ and integrate from $t=1 $ to $ t=T$, utilize our assumption on initial data, then obtain
$$
\begin{aligned}
\int_{1}^{T} t^{\kappa-1} Q(t) d t & \leq \int_{1}^{T} t^{\kappa-1}\left(\frac{\lambda}{t} \int_{0}^{t} Q\left(t^{\prime}\right) d t^{\prime}\right) d t+C_{\kappa} \int_{\mathbb{R}^{d}} \min \left\{|x|,|x|^{\kappa}\right\} e(x, 0) d x \\
& \leq \frac{\lambda}{1-\kappa} \int_{0}^{T} \min \left\{\left(t^{\prime}\right)^{\kappa-1}, 1\right\} Q\left(t^{\prime}\right) d t^{\prime}+C_{\kappa} \int_{\mathbb{R}^{d}} \min \left\{|x|,|x|^{\kappa}\right\} e(x, 0) d x \\
& \leq \frac{\lambda}{1-\kappa} \int_{1}^{T}\left(t^{\prime}\right)^{\kappa-1} Q\left(t^{\prime}\right) d t^{\prime}+C_{\kappa} \int_{\mathbb{R}^{d}} \min \left\{|x|,|x|^{\kappa}\right\} e(x, 0) d x+C_{k} E.
\end{aligned}
$$
Here$$\frac{\lambda}{1-\kappa}=\frac{4-(d-1)(p-1)}{2(1-\kappa)}<1,$$
since we have assumed $\kappa<\frac{(d-1)(p-1)-2}{2}$. Therefore we have $$
\int_{1}^{T} t^{\kappa-1} Q(t) d t \lesssim _{p, \kappa} \int_{\mathbb{R}^{d}} \min \left\{|x|,|x|^{\kappa}\right\} e(x, 0) d x+E.
$$
Because neither the right hand side nor the implicit constant here depends on $T$, we make $T\to+\infty$ to conclude$$
\int_{1}^{\infty} t^{\kappa-1} Q(t) d t<+\infty.
$$
Combining this with the fact $Q(t)\lesssim E$, we have
$$\int_{0}^{\infty} t^{\kappa-1} Q(t) d t<+\infty.$$
We may multiply both sides of the recurrence formula by $t^{\kappa}$:
$$t^{\kappa}Q(t) \leq \lambda \int_{0}^{t} t^{\kappa-1} Q\left(t^{\prime}\right) d t^{\prime}+2 \int_{\mathbb{R}^{d}} \min \{|x| / t^{\kappa-1}, t^\kappa\} e(x, 0) d x.$$
Finally we apply dominated convergence theorem to finish the proof of Theorem 1.1.
\begin{remark}
When $d>3$, we have $$
\left|u_{r}+u_{t}\right|^{2} \lesssim_{1}\left|u_{r}+\frac{(d-1)u}{2|x|}+u_{t}\right|^{2}+\frac{(d-1)(d-3)}{8}\frac{|u|^{2}}{|x|^{2}},\qquad\qquad \text { if }\,\,|x|>t .
$$ 
But the term $\frac{(d-1)(d-3)}{8}\frac{|u|^{2}}{|x|^{2}}=0$ when $d=3$, so the inequality above does not hold. We can redefine
$$
\begin{aligned}
Q(t)=& \frac{1}{p+1} \int_{\mathbb{R}^{d}}\left(|u(x, t)|^{p+1}+|u(x,-t)|^{p+1}\right) d x+c_{1} \sum_{\pm} \int_{|x|<t} \frac{t-|x|}{t}\left(|\nabla u(x, \pm t)|^{2}+\left|u_{t}(x, \pm t)\right|^{2}\right) d x \\
& +c_{2} \sum_{\pm} \int_{|x|<t}\left(\left|\left(u_{r} \pm u_{t}\right)(x, \pm t)\right|^{2}+|\nabla\mkern-13mu/ u(x, \pm t)|^{2}\right) d x\\
& +c_{2} \sum_{\pm} \int_{|x|>t}\left(\left|u_{r}+\frac{(d-1) u}{2|x|} \pm u_{t}\right|^{2}+|\nabla\mkern-13mu/ u(x, \pm t)|^{2}\right) d x.
\end{aligned}
$$
We also have $$\limsup_{t \rightarrow+\infty} Q(t)=0,$$ and $$t^{\kappa}Q(t) \leq \lambda \int_{0}^{t} t^{\kappa-1} Q\left(t^{\prime}\right) d t^{\prime}+2 \int_{\mathbb{R}^{d}} \min \{|x| / t^{\kappa-1}, t^\kappa\} e(x, 0) d x.$$
Because $$\left|u_{r}+u_{t}\right|^{2} \lesssim_{1}\left|u_{r}+\frac{(d-1)u}{2|x|}+u_{t}\right|^{2}+c\frac{|u|^{2}}{|x|^{2}},\qquad\qquad \text { if }\,\,|x|>t .$$
We can finish the proof of Theorem 1.1 by an estimate.
$$
\begin{aligned}
\int_{|x|>t} \frac{|u|^{2}}{|x|^{2}} d x & \lesssim \left(\int_{|x|>t}\left(|u|^{2}\right)^{\frac{p+1}{2}} d x\right)^{\frac{2}{p+1}}\left(\int_{|x|>t}\left(|x|^{-2}\right)^{\frac{p+1}{p-1}} d x\right)^{\frac{p-1}{p+1}} \\
& \lesssim_{p} t^{\frac{-2 \kappa}{p+1}+\frac{p-5}{p+1}}.
\end{aligned}
$$
When $d=3$, we have $p\in(2,3)$ and the inequality $$t^{\frac{-2\kappa}{p+1}+\frac{p-5}{p+1}}\ll t^{-\kappa},$$
thus
$$\int_{|x|>t} \frac{|u|^{2}}{|x|^{2}} d x \ll t^{-\kappa}.$$

\end{remark}
\section{Scattering Theory}
\subsection{Transformation to 1D}
In order to take full advantage of our radial assumption, we use the following transformation: if $u$ is a radial solution to (CP1), then $w(r, t)=r^{\frac{d-1}{2}} u(r, t)$, where $|x|=r$, is a solution to one-dimensional wave equation $$
\left(\partial_{t}+\partial_{r}\right)\left(w_{t}-w_{r}\right)=\partial_{t}^{2} w-\partial_{r}^{2} w=-\frac{(d-1)(d-3)}{4} r^{\frac{d-5}{2}} u-r^{\frac{d-1}{2}}|u|^{p-1} u
.$$
We define$$v_{+}(r,t)=w_{t}(r,t)-w_{r}(r,t);\qquad v_{-}(r,t)=w_{t}(r,t)+w_{r}(r,t).$$ 
A simple calculation shows that $v_{\pm}$ satisfy the equation
$$
\left(\partial_{t} \pm \partial_{r}\right) v_{\pm}(r, t)=\partial_{t}^{2} w-\partial_{r}^{2} w=-\frac{(d-1)(d-3)}{4} r^{\frac{d-5}{2}} u- r^{\frac{d-1}{2}}|u|^{p-1} u
.$$This gives variation of $v_{\pm}$ along characteristic lines $t \pm r= Const$.
$$
v_{+}\left(t_{2}-\eta, t_{2}\right)-v_{+}\left(t_{1}-\eta, t_{1}\right)=\int_{t_{1}}^{t_{2}} f(t-\eta, t) d t, \quad t_{2}>t_{1}>\eta;
\eqno{(15)}$$
$$
v_{-}\left(s-t_{2}, t_{2}\right)-v_{-}\left(s-t_{1}, t_{1}\right)=\int_{t_{1}}^{t_{2}} f(s-t, t) d t, \quad t_{1}<t_{2}<s .\eqno{(16)}
$$Here the function $f(r, t)$ is defined by$$
f(r, t)=-\frac{(d-1)(d-3)}{4} r^{\frac{d-5}{2}} u(r, t)- r^{\frac{d-1}{2}}|u|^{p-1} u(r, t)
.$$
Then we give the upper bounds of the integral above. According to Lemma 2.1 we have
$$
\begin{aligned}
& \int_{t_{1}}^{t_{2}}(t-\eta)^{\frac{d-5}{2}}|u(t-\eta, t)| d t \\
\leq &\left\{\int_{t_{1}}^{t_{2}}\left[(t-\eta)^{\frac{d-3}{2}}|u(t-\eta, t)|\right]^{2} d t\right\}^{1 / 2}\left\{\int_{t_{1}}^{t_{2}}\left[(t-\eta)^{-1}\right]^{2} d t\right\}^{1 / 2} \\
\leq &\left\{\int_{\eta}^{\infty}(t-\eta)^{d-3}|u(t-\eta, t)|^{2} d t\right\}^{1 / 2}\left(t_{1}-\eta\right)^{-1 / 2} \\
\lesssim &_{d} E^{1 / 2}\left(t_{1}-\eta\right)^{-1 / 2}.
\end{aligned}
$$
In addition we have
$$
\begin{aligned}
& \int_{t_{1}}^{t_{2}}(t-\eta)^{\frac{d-1}{2}}|u(t-\eta, t)|^{p} d t \\
\leq &\left\{\int_{t_{1}}^{t_{2}}\left[(t-\eta)^{\frac{(d-1) p}{p+1}}|u(t-\eta, t)|^{p}\right]^{\frac{p+1}{p}} d t\right\}^{\frac{p}{p+1}}\left\{\int_{t_{1}}^{t_{2}}\left[(t-\eta)^{-\frac{(d-1)(p-1)}{2(p+1)}}\right]^{p+1} d t\right\}^{\frac{1}{p+1}} \\
\leq &\left\{\int_{\eta}^{\infty}(t-\eta)^{d-1}|u(t-\eta, t)|^{p+1} d t\right\}^{\frac{p}{p+1}}\left\{\int_{t_{1}}^{t_{2}}(t-\eta)^{-\frac{(d-1)(p-1)}{2}} d t\right\}^{\frac{1}{p+1}} \\
\lesssim &_{d}  E^{\frac{p}{p+1}}\left(t_{1}-\eta\right)^{-\frac{(d-1)(p-1)-2}{2(p+1)}} .
\end{aligned}
$$
we combine these estimates above with (15) and (16) to obtain
\begin{lemma}
Let $u$ be a radial solution wave equation with a finite energy  E. Then we have
$$
\left|v_{+}\left(t_{2}-\eta, t_{2}\right)-v_{+}\left(t_{1}-\eta, t_{1}\right)\right| \lesssim_{d} E^{1 / 2}\left(t_{1}-\eta\right)^{-1 / 2}+ E^{\frac{p}{p+1}}\left(t_{1}-\eta\right)^{-\beta(d, p)/2} ;
$$
$$
\left|v_{-}\left(s-t_{2}, t_{2}\right)-v_{-}\left(s-t_{1}, t_{1}\right)\right| \lesssim_{d} E^{1 / 2}\left(s-t_{2}\right)^{-1 / 2}+ E^{\frac{p}{p+1}}\left(s-t_{2}\right)^{-\beta(d, p)/2}.
$$Here $\beta(d,p)=\frac{(d-1)(p-1)-2}{p+1}$.
\end{lemma}
\subsection{Scattering by energy decay}In this section we prove Theorem 1.2. Let us recall the lemma 5.1, we obtain there exists a function $g_{+}(\eta) \in L^{2}(\mathbb{R})$ with $\left\|g_{+}\right\|_{L^{2}(\mathbb{R})}^{2} \leq E / c_{d}$, so that$$
v_{+}(t-\eta, t) \rightarrow 2 g_{+}(\eta) \quad \text { in } L_{l o c}^{2}(\mathbb{R}), \quad \text { as } t \rightarrow+\infty.
$$ The asymptotic behaviour of $v_{-}$ is similar as $t\rightarrow-\infty$$$
v_{-}(s-t, t) \rightarrow 2 g_{-}(s) \quad \text { in } L_{l o c}^{2}(\mathbb{R}), \quad \text { as } t \rightarrow-\infty,
$$let $t_{2}\to+\infty$ in the first inequality of lemma 5.1, we have $$
\left|2 g_{+}(\eta)-v_{+}(t-\eta, t)\right| \lesssim_{d,p,E}(t-\eta)^{-\beta(d, p)/2}, \quad \eta<t-1.
$$We apply a change of variable $r=t-\eta$ and rewrite this in the form $$
\left|v_{+}(r, t)-2 g_{+}(t-r)\right| \lesssim_{d,p,E} r^{-\beta(d, p)/2}, \quad r>1.
$$
Similarly we have$$
\left|v_{-}(r, t)-2 g_{-}(t+r)\right| \lesssim_{d,p,E} r^{-\beta(d, p)/2}, \quad r>1.
$$These gives the following upper limits $$
\limsup _{t \rightarrow+\infty} \int_{t-c \cdot t^{ \beta(d, p)}}^{t+R}\left(\left|v_{+}(r, t)-2 g_{+}(t-r)\right|^{2}+\left|v_{-}(r, t)-2 g_{-}(t+r)\right|^{2}\right) d r \lesssim_{d,p,E} c.
\eqno{(17)}$$
We ignore $g_{-}(t+r)$ in the upper limits above because$$
\lim _{t \rightarrow+\infty} \int_{0}^{\infty}\left|g_{-}(t+r)\right|^{2} d r=\lim _{t \rightarrow+\infty} \int_{t}^{\infty}\left|g_{-}(s)\right|^{2} d s=0.
$$
We recall $v_{\pm}=w_{t} \mp w_{r}$ and rewrite the upper limits above in term of $w$ $$
\limsup _{t \rightarrow+\infty} \int_{t-ct^{\beta(d,p)}}^{t+R}\left(\left|w_{r}(r, t)+g_{+}(t-r)\right|^{2}+\left|w_{t}(r, t)-g_{+}(t-r)\right|^{2}\right) d r \lesssim _{d,p,E} c.
$$
Next we utilize the identities $r^{\frac{d-1}{2}} u_{r}=w_{r}-(d-1) r^{\frac{d-3}{2}} u / 2, r^{\frac{d-1}{2}} u_{t}=w_{t}$ and a direct consequence of the pointwise estimate $|u(r, t)| \lesssim_{d, E} r^{-\frac{2(d-1)}{p+3}}$( lemma 2.1 ) to conclude$$
\limsup _{t \rightarrow+\infty} \int_{t-c \cdot t^{ \beta(d, p)}}^{t+R}\left(\left|r^{\frac{d-1}{2}} u_{r}(r, t)+g_{+}(t-r)\right|^{2}+\left|r^{\frac{d-1}{2}} u_{t}(r, t)-g_{+}(t-r)\right|^{2}\right) d r \lesssim_{d,p,E} c. \eqno{(18)}
$$
By lemma 1.2 (radiation fields), there exists a radial free wave $\tilde{u}^{+}$, so that
$$
\lim_{t \rightarrow+\infty} \int_{0}^{\infty}\left(\left|r^{\frac{d-1}{2}} \tilde{u}_{r}^{+}(r,t)+g_{+}(t-r)\right|^{2}+\left|r^{\frac{d-1}{2}} \tilde{u}_{t}^{+}(r,t)-g_{+}(t-r)\right|^{2}\right) d r =0. \eqno{(19)}
$$
Therefore we have
$$
\limsup _{t \rightarrow+\infty} \int_{t-c \cdot t^{\beta(d,p)}}^{t+R} r^{d-1} \left(\left|u_{t}(r, t)-\tilde{u}_{t}^{+}(r, t)\right|^{2}+\left|u_{r}(r, t)-\tilde{u}_{r}^{+}(r, t)\right|^{2}\right) d r \lesssim_{d,p,E} c.
\eqno{(20)}$$
Finite speed of propagation of energy implies 
$$
\limsup_{R\to \infty \, t>0}\int_{t+R}^{\infty}r^{d-1}\left(\left|u_{t}(r, t)-\tilde{u}_{t}^{+}(r, t)\right|^{2}+\left|u_{r}(r, t)-\tilde{u}_{r}^{+}(r, t)\right|^{2}\right) d r=0.\eqno{(21)}
$$
We combine $(20)$ with $(21)$ and obtain
$$
\limsup _{t \rightarrow+\infty} \int_{t-c \cdot t^{\beta(d,p)}}^{+\infty} r^{d-1}\left(\left|u_{t}(r, t)-\tilde{u}_{t}^{+}(r, t)\right|^{2}+\left|u_{r}(r, t)-\tilde{u}_{r}^{+}(r, t)\right|^{2}\right) d r \lesssim_{d,p,E} c. \eqno{(22)}
$$
Finally we consider the region $\left\{x:|x|<t-c \cdot t^{\beta(p)}\right\}$. We utilize the conclusion of Theorem 1.1.(c), and obtain
$$
\lim _{t \rightarrow+\infty} \int_{|x|<t-c \cdot t^{\beta(d,p)}} e(x, t) d x \lesssim_{c} \lim _{t \rightarrow+\infty} t^{\frac{(2-d)p+(d+2)}{p+1}} \int_{|x|<t} \frac{t-|x|}{t} e(x, t) d x=0.
$$
Please note that in the sub-conformal range our assumption $p>\frac{3-d+2\sqrt{d^2-d+1}}{d-1}$ guarantees that $\frac{(2-d)p+(d+2)}{p+1}<\frac{(d-1)(p-1)-2}{2}$.\\ 
We also have$$
\lim _{t \rightarrow+\infty}\int_{|x|<t-ct^{\beta(p)}}\left(\left|\nabla \tilde{u}^{+}(x, t)\right|^{2}+\left|\tilde{u}_{t}^{+}(x, t)\right|^{2}\right) d x=0.
$$Combining these two limits we obtain$$
\lim _{t \rightarrow+\infty}\int_{|x|<t-ct^{\beta(d,p)}} \left(\left|\nabla \tilde{u}^{\pm}(x, t)-\nabla u(x, t)\right|^{2}+\left|\tilde{u}_{t}^{\pm}(x, t)-u_{t}(x, t)\right|^{2}\right) d x=0.
$$
We combine this with stronger exterior scattering to conclude $$
\limsup _{t \rightarrow+\infty} \int_{\mathbb{R}^{d}}\left(\left|\nabla \tilde{u}^{\pm}(x, t)-\nabla u(x, t)\right|^{2}+\left|\tilde{u}_{t}^{\pm}(x, t)-u_{t}(x, t)\right|^{2}\right) d x \lesssim_{d,p,E}  c.
$$
We make $c\rightarrow0^{+}$ and finish the proof.
\section*{Acknowledgement}
The second author is financially supported by National Natural Science Foundation of China Projects 12071339, 11771325. 

\end{document}